\documentclass[10pt,leqno,letterpaper]{article}
\usepackage[utf8]{inputenc}
\usepackage{amsmath}
\usepackage{amsfonts}
\usepackage{amssymb}
\usepackage{graphicx}
\usepackage{verbatim}
\usepackage{mathrsfs}
\usepackage{upref,amsthm,amsxtra,exscale}
\usepackage{cite}
\usepackage[colorlinks=true,urlcolor=blue,
citecolor=red,linkcolor=blue,linktocpage,pdfpagelabels,
bookmarksnumbered,bookmarksopen]{hyperref}

\usepackage{fullpage}

\newtheorem{theorem}{Theorem}[section]

\newtheorem{remark}[theorem]{Remark}

\newtheorem{lemma}[theorem]{Lemma}
\newtheorem{proposition}[theorem]{Proposition}
\newtheorem{definition}[theorem]{Definition}

\newtheorem{questions}[theorem]{Questions}

\numberwithin{equation}{section}

\def\r{\mathbb{R}}
\def\rn{\mathbb{R}^N}
\def\z{\mathbb{Z}}
\def\zl{\z/\ell\z}
\def\s1{\mathbb{S}^1}
\def\n{\mathbb{N}}
\def\cc{\mathbb{C}}
\def\eps{\varepsilon}
\def\rh{\rightharpoonup}
\def\io{\int_{\Omega}}
\def\irn{\int_{\r^N}}
\def\vp{\varphi}

\def\vr{\varrho}
\def\o{\Omega}

\def\bf{\boldsymbol}

\def\cC{\mathcal{C}}
\def\cD{\mathcal{D}}

\def\cJ{\mathcal{J}}

\def\cM{\mathcal{M}}
\def\cN{\mathcal{N}}

\def\cP{\mathcal{P}}

\def\cW{\mathcal{W}}

\def\supp{\text{supp}}
\def\bar{\overline}
\def\what{\widehat}

\def\i{\mathrm{i}}
\def\dist{\mathrm{dist}}
\author{Mónica Clapp\footnote{M. Clapp was supported by CONACYT (Mexico) through the research grant A1-S-10457.}, \ Jorge Faya \footnote{J. Faya was supported  by ANID (Chile) through the project FONDECYT/Iniciacion 11190423}  \ and \ Alberto Saldaña\footnote{A. Saldaña was supported by UNAM-DGAPA-PAPIIT (Mexico) grant IA100923  and by CONACYT (Mexico) grant A1-S-10457 and  CBF2023-2024-116.}\ \footnote{Corresponding Author, any correspondence should be addressed to \texttt{alberto.saldana@im.unam.mx}.}}
\title{Optimal pinwheel partitions for the Yamabe equation}
\date{}

\begin{document}
\maketitle

\begin{abstract}
We establish the existence of an optimal partition for the Yamabe equation in $\rn$ made up of mutually linearly isometric sets, each of them invariant under the action of a group of linear isometries.

To do this, we establish the existence of a solution to a weakly coupled competitive Yamabe system, whose components are invariant under the action of the group, and each of them is obtained from the previous one by composing it with a linear isometry. We show that, as the coupling parameter goes to $-\infty$, the components of the solutions segregate and give rise to an optimal partition that has the properties mentioned above.

Finally, taking advantage of the symmetries considered, we establish the existence of infinitely many sign-changing solutions for the Yamabe equation in $\rn$ that are different from those previously found by W.Y. Ding, and del Pino, Musso, Pacard and Pistoia.
\smallskip

\smallskip

\end{abstract}

\section{Introduction}

A fundamental question in differential geometry is whether every closed Riemannian manifold $M$ with given metric $g$ possesses a Riemannian metric $\what g$ conformally equivalent to $g$ (i.e., $\what g$ is a positive multiple of $g$) with constant scalar curvature. In dimension $2$ the answer is positive; it is a consequence of the classical uniformization theorem of complex analysis. The question whether this is also true in higher dimensions was posed by Yamabe \cite{y} and it is called the Yamabe problem. Providing a positive answer in dimension $\geq 3$ is not easy and it took 25 years until this problem was finally solved by the combined efforts of Yamabe \cite{y}, Trudinger \cite{t}, Aubin \cite{a} and Schoen \cite{sch}. A detailed discussion may be found in \cite{lp}. 

Solving the Yamabe problem is equivalent with establishing the existence of a positive solution to an elliptic equation on $M$ called the Yamabe equation. If $u$ is a least energy nodal solution to the Yamabe equation, i.e., a solution that changes sign, then it gives rise to a partially defined metric that solves the Yamabe problem in each nodal domain.  Ammann and Humbert called it a generalized metric \cite{ah}. It satisfies an optimality condition, namely, the domains of definition of the metric form an optimal partition for the Yamabe equation. 

An optimal partition for the Yamabe equation on $M$ is a finite collection of nonempty pairwise disjoint open subsets of $M$, whose closures cover $M$, such that the Yamabe equation with Dirichlet boundary condition in each one of these sets has a least energy solution. Furthermore, the following optimality condition is satisfied: the sum of the energies of these solutions is minimal with respect to any other partition with the same number of elements satisfying the same properties. Finding a partition that minimizes some cost function is, in general, not an easy problem and it appears in many contexts, both in mathematics and in other fields as well. We refer the reader to \cite{bb,ta} for a survey on this topic.

Not every manifold admits an optimal partition, for instance, as we explain below, the standard $N$-sphere $\mathbb{S}^N$ does not. Conditions for the existence of an optimal partition for the Yamabe equation were recently established in \cite{cpt}.

In this paper we focus on the Yamabe equation in $\rn$, which is conformally equivalent to $\mathbb{S}^N$ via the stereographic projection. Namely, we consider the problem
\begin{equation} \label{eq:yamabe0}
\begin{cases}
-\Delta u = |u|^{2^*-2}u, \\
u\in D^{1,2}(\rn),
\end{cases}
\end{equation}
where $N\geq 4$, $2^*:=\frac{2N}{N-2}$ is the critical Sobolev exponent, and $D^{1,2}(\rn)$ is the Sobolev space of functions in $L^{2^*}(\rn)$ whose weak derivatives are in $L^2(\rn)$. 

This equation does not admit an optimal partition in the sense defined above, because, in any open subset $\o$ of $\rn$ whose closure has nonempty complement, the Dirichlet problem
$$-\Delta u = |u|^{2^*-2}u \text{ \ in \ } \o, \qquad u=0 \text{ \ on \ }\partial\o,$$
does not have a least energy solution \cite[Theorem III.1.2]{s}. This is connected to the well-known invariance under rescalings of this problem, which causes a lack of compactness. But if one asks that, in addition, the sets of the partition are invariant under the action of some group of transformations, then the question has better chances.  For example, an optimal partition for \eqref{eq:yamabe0} whose elements are invariant under the conformal action of $O(m)\times O(n)$, $m+n=N+1$, $m,n\geq 2$, induced by the obvious action on $\mathbb{S}^N$ via the stereographic projection, was exhibited in \cite{css}.

It is interesting to explore whether there are optimal partitions for \eqref{eq:yamabe0} having other shapes and properties. For instance,
\begin{center}
\emph{are there optimal partitions whose elements are mutually linearly isometric?}
\end{center}
namely, optimal partitions where all the subsets have the same shape and size. We note that the elements of the partitions displayed in \cite{css} are not even homeomorphic to each other.  In this paper, we give a positive answer. To do this, the main challenge is to find two compatible types of symmetries: one that prevents the lack of compactness of the critical problems ($G$ in the notation below), and the other that ensures that all the subsets in the partition have the same shape and size ($\vr$ in the notation below).  To be more precise, let us describe the kind of partitions that we are interested in.

Let $G$ be a closed subgroup of the group $O(N)$ of linear isometries of $\rn$, $\vr\in O(N)$ and $\ell\geq 2$. Recall that an open subset $\o$ of $\rn$ is called $G$-invariant if $gx\in\o$ for every $g\in G$ and $x\in\o$, and a function $u:\o\to\r$ is called $G$-invariant if $u(gx)=u(x)$ for every $g\in G$ and $x\in\o$. An \textbf{optimal $(G,\vr,\ell)$-partition for \eqref{eq:yamabe0}} is an $\ell$-tuple $(\o_1,\ldots,\o_\ell)$ of subsets of $\rn$ with the following properties:
\begin{itemize}
\item[$(P_1)$] $\o_i\neq\emptyset$, \ $\o_i\cap\o_j=\emptyset$ if $i\neq j$, \ $\rn=\bigcup_{i=1}^\ell\overline{\o}_i$ \ and \ $\o_i$ is open and $G$-invariant.
\item[$(P_2)$] For each  $i=1,\ldots,\ell$, the critical problem in $\o_i$
\begin{equation} \label{eq:critical_problem}
\begin{cases}
-\Delta u = |u|^{2^*-2}u, \\
u\in D^{1,2}_0(\o_i), \\
u\text{ is }G\text{-invariant},
\end{cases}
\end{equation}
has a nontrivial least energy solution $u_i$.
\item[$(P_3)$] $\vr(\o_2)=\o_1, \ \ldots \ , \ \vr(\o_\ell)=\o_{\ell-1}, \ \ \vr(\o_1)=\o_\ell.$
\item[$(P_4)$] The sum of the energies of the $u_i$'s is minimal with respect to any other partition of $\rn$ satisfying $(P_1)$, $(P_2)$ and $(P_3)$.
\end{itemize}

Note that each element $\o_i$ of an optimal $(G,\vr,\ell)$-partition for \eqref{eq:yamabe0} must be unbounded, and that it cannot be strictly star-shaped, otherwise the critical problem \eqref{eq:critical_problem} would not have a nontrivial solution, see \cite[Theorem III.1.3]{s}. Furthermore, $G$ cannot be the trivial group because, as we mentioned above,  \eqref{eq:critical_problem} does not have a nontrivial least energy solution in this case.

In this setting, we now follow the strategy presented in \cite{ctv}: in order to obtain an optimal $(G,\vr,\ell)$-partition we consider the Yamabe system
\begin{equation} \label{eq:system_rn}
\begin{cases}
-\Delta u_i = |u_i|^{2p-2}u_i+\sum\limits_{\substack{j=1\\j\neq i}}^\ell\beta|u_j|^p|u_i|^{p-2}u_i, \\
u_i\in D^{1,2}(\rn),\qquad i=1,\ldots,\ell,
\end{cases}
\end{equation}
where $p:=\frac{N}{N-2}$ and $\beta <0$, and we look for a nontrivial least energy solution $\bf u=(u_1,\ldots,u_\ell)$ that satisfies
\begin{itemize}
\item[$(S_1)$] $u_j$ is $G$-invariant for every $j=1,\ldots,\ell$, 
\item[$(S_2)$] $u_{j+1}=u_j\circ\vr$ \ for every $j=1,\ldots,\ell$, \ where \ $u_{\ell+1}:=u_1$.
\end{itemize}
As shown in \cite{ctv}, the components of least energy solutions to systems of this type separate spatially as $\beta\to-\infty$ and are bound to give rise to an optimal partition. This property can be observed in some physical phenomena. Systems like \eqref{eq:system_rn} describe, for instance, the behavior of standing waves for a mixture of Bose-Einstein condensates of hyperfine states that overlap in space. As the repelling force $\beta$ between different states increases, the components of least energy solutions to  \eqref{eq:system_rn} become spatially segregated. This phenomenon is
called phase separation and has been studied extensively (see the survey \cite{ta}).

Now, \eqref{eq:system_rn} does not necessarily have a least energy solution satisfying $(S_1)$ and $(S_2)$. For instance, if $G$ is the trivial group and $\vr$ is the identity, the system reduces to the single equation
$$-\Delta u = (1+(\ell-1)\beta)|u|^{2p-2}u,\qquad u\in D^{1,2}(\rn),$$
which does not have a nontrivial solution if $1+(\ell-1)\beta<0$. 

In the following, we exhibit a suitable symmetry group $G$ and a compatible linear isometry $\vr$ such that \eqref{eq:system_rn} admits a solution that gives rise to an optimal $(G,\vr,\ell)$-partition for \eqref{eq:yamabe0}. We write $\rn\equiv\cc^2\times\r^{N-4}$ and a point in $\rn$ as $x=(z,y)$ with $z=(z_1,z_2)\in\cc\times\cc$ and $y\in\r^{N-4}$. Let $\tau:\cc^2\to\cc^2$ be given by
\begin{equation*}
\tau(z_1,z_2):=(-\overline{z}_2,\overline{z}_1),
\end{equation*}
where $\overline{z}$ is the complex conjugate of $z$. We define $G$ to be the group whose elements are the complex numbers on the unit circle $\s1$ acting on $\rn$ by
\begin{equation}\label{eq:G}
gx:=(gz_1,\overline{g}z_2,y),\qquad\text{for \ }g\in G, \ x=(z_1,z_2,y)\in\cc\times\cc\times\r^{N-4},
\end{equation}
and take $\vr:=\vr_\ell\in O(N)$ to be
\begin{equation}\label{eq:rho}
\vr_\ell(z,y):=\left(\Big(\cos\frac{\pi}{\ell}\Big)z+\Big(\sin\frac{\pi}{\ell}\Big)\tau z,\,y\right)\qquad\text{for \ }(z,y)\in\cc^2\times\r^{N-4}.
\end{equation}
Inspired by \cite{cp}, an optimal $(G,\vr_\ell,\ell)$-partition for these particular choices is called an \textbf{optimal $\ell$-pinwheel partition for \eqref{eq:yamabe0}} and a solution to the system \eqref{eq:system_rn} that satisfies $(S_1)$ and $(S_2)$ for these data is called a \textbf{pinwheel solution}.

We obtain the following existence result.

\begin{theorem} \label{thm:existence}
If $N\geq 4$ the system \eqref{eq:system_rn} has a nontrivial least energy pinwheel solution $\bf u=(u_1,\ldots,u_\ell)$ that satisfies $u_j\geq 0$ for every $j=1,\ldots,\ell$.
\end{theorem}

The precise meaning of least energy pinwheel solution is given in Definition \ref{def:le}. 

Note that the system \eqref{eq:system_rn} is invariant under dilations and translations. So, if $\eps>0$, $\xi\in\{0\}\times\r^{N-4}$ and $\bf u=(u_1,\ldots,u_\ell)$ is a pinwheel solution, then $\bf w=(w_1,\ldots,w_\ell)$ with $w_i(x):=\eps^\frac{N-2}{2}u_i(\eps x+\xi)$  is also a pinwheel solution with the same energy value as $\bf u$. This fact produces a lack of compactness of the variational functional associated with the system. Also, the components of least energy solutions to \eqref{eq:system_rn} could blow up as $\beta\to-\infty$. Therefore, to obtain an optimal partition, the solutions must be properly reparameterized beforehand. We prove the following result.

\begin{theorem} \label{thm:main2} 
Assume $N\geq 4$. Let $\beta_k<0$ satisfy $\beta_{k}\rightarrow -\infty$ and $\bf u_k=(u_{1,k},\ldots,u_{\ell,k})$ be a nontrivial least energy pinwheel solution to the system \eqref{eq:system_rn} with $\beta=\beta_{k}$ such that $u_{i,k}\geq 0$. Then, after passing to a subsequence, there exist $\eps_k\in(0,\infty)$ and $\xi_k\in\{0\}\times\r^{N-4}$ such that the function $\bf w_k=(w_{1,k},\ldots,w_{\ell,k})$ given by
$$w_{i,k}(x):=\eps_k^\frac{N-2}{2}u_{i,k}(\eps_kx+\xi_k),\qquad i=1,\ldots,\ell,$$
is a nontrivial least energy pinwheel solution to the system \eqref{eq:system_rn} for $\beta=\beta_{k}$ that satisfies
	\begin{enumerate}		
		\item [$(i)$] $\bf w_{k}\rightarrow \bf w_{\infty}=(w_{1,\infty},\dots, w_{\ell,\infty})$ strongly in $(D^{1,2}(\rn))^\ell$, \ $w_{j,\infty}\neq 0$ and $w_{j,\infty}\geq 0$ for every $j$, and
		$$ w_{j,\infty}w_{i,\infty}=0 \quad  \mbox{and} \quad	\lim_{k\to\infty}\int_{\mathbb{R}^{N}}\beta_{k}w_{i,k}^{p}w_{j,k}^{p}=0 \quad \mbox{ whenever \ } i\neq j. $$ 
		\item [$(ii)$] $w_{j,\infty}\in \cC^{0}(\mathbb{R}^{N})$ and the restriction of $w_{j,\infty}$ to the set $\Omega_{j}:=\{x\in\mathbb{R}^{N}: w_{j,\infty}>0 \}$ is a least  energy solution to the problem \eqref{eq:critical_problem} with $G$ as in \eqref{eq:G}.
		\item [$(iii)$] $\mathbb{R}^{N}\smallsetminus \bigcup_{j=1}^{\ell}\Omega_{j}=\mathscr{R}\cup \mathscr{S}$, where $\mathscr{R}\cap \mathscr{S}=\emptyset$, $\mathscr{R}$ is an $(N-1)$-dimensional $\cC^{1,\alpha}$-submanifold of $\mathbb{R}^{N}$ and $\mathscr{S}$ is a closed subset of  $\mathbb{R}^{N}$ with Hausdorff measure $\leq N-2$. Moreover, if $\xi\in \mathscr{R}$, there exist $i,j$ such that 
		\begin{equation*}
			\lim_{x\rightarrow \xi^{+}}|\nabla w_{i,\infty}(x)|= \lim_{x\rightarrow \xi^{-}}|\nabla w_{j,\infty}(x)|\neq 0,
		\end{equation*}
		where $x\rightarrow \xi^{\pm}$ are the limits taken from opposite sides of $ \mathscr{R}$ and, if $\xi\in \mathscr{S}$, then 
		\begin{equation*}
			\lim_{x\rightarrow \xi}|\nabla w_{j,\infty}(x)|= 0 \quad \mbox{ for every } \ j=1,\dots, \ell.
		\end{equation*}	
		\item[$(iv)$] $(\Omega_{1}, \dots, \Omega_{\ell})$ is an optimal $\ell$-pinwheel partition for the Yamabe problem \eqref{eq:yamabe0}.
		\item [$(v)$] If $\ell=2$, the function $\widehat{w}=w_{1,\infty}-w_{2,\infty}$ is a sign-changing solution to the Yamabe problem \eqref{eq:yamabe0} that satisfies 
		\begin{equation*}
\widehat{w}(gx)=\widehat{w}(x) \quad\text{and}\quad \widehat{w}(\vr_2x)= -\widehat{w}(x) \qquad\text{for all \ }g\in G, \ x\in\rn,
		\end{equation*}
with $G$ and $\vr_2$ as defined in \eqref{eq:G} and \eqref{eq:rho}. Furthermore, $\widehat{w}$ has least energy among all nontrivial solutions to \eqref{eq:yamabe0} with these properties.
	\end{enumerate}
\end{theorem}

Many different kinds of sign-changing solutions are known for the Yamabe problem \eqref{eq:yamabe0}, see \cite{bsw, c, cpw, d, dmpp, fp}. The one given by Theorem \ref{thm:main2}$(v)$ was obtained in \cite{c}. The following result provides infinitely many solutions of this last type. It is a consequence of \cite[Corollary 3.4]{c}.

\begin{theorem}\label{thm:nodal}
Let $N\geq 4$.
\begin{itemize}
\item[$(i)$] For each $\ell\in\n$ there exists a sign-changing solution $\widehat{w}_\ell$ of the Yamabe problem \eqref{eq:yamabe0} such that
\begin{equation*}
\widehat{w}_\ell(gx)=\widehat{w}_\ell(x) \quad\text{and}\quad \widehat{w}_\ell(\vr_{\ell} x)= -\widehat{w}_\ell(x) \qquad\text{for all \ }g\in G, \ x\in\rn,
\end{equation*}
with $G$ and $\vr_{\ell}$ as defined in \eqref{eq:G} and \eqref{eq:rho}. Furthermore, $\widehat{w}_\ell$ has least energy among all nontrivial solutions to \eqref{eq:yamabe0} with these properties.
\item[$(ii)$] If $\ell=nm$ and $n$ is even, then $\widehat{w}_\ell$ and $\widehat{w}_m$ are different and one is not obtained by translation and dilation of the other. In particular,
$$\widehat{w}_1,\,\widehat{w}_2,\ldots,\widehat{w}_{2^k},\ldots$$
is an infinite sequence of sign-changing solutions of \eqref{eq:yamabe0} that are not equivalent under translation and dilation.
\end{itemize}
\end{theorem}

These solutions are different from those given by Ding in \cite{d} whose nodal regions are not homeomorphic to each other, see \cite{css}. They are also different from those found by del Pino, Musso, Pacard and Pistoia in \cite{dmpp}, which look like a positive bubble centered at the origin surrounded by a large number of highly dilated negative bubbles placed on the vertices of a regular polygon.

Solutions to the Yamabe system \eqref{eq:system_rn} that are invariant under the conformal action of $O(m)\times O(n)$, $m+n=N+1$, were found in \cite{cp2,cs}, and the behavior, as $\beta\to-\infty$, of least energy solutions with this property was described in \cite{cp2,css}. Solutions to \eqref{eq:system_rn} resembling those given in \cite{dmpp} for the single equation were recently obtained in \cite{glw,cmp} using a Ljapunov-Schmidt reduction procedure. This procedure can only be applied in dimensions $N=3,4$, because the coupling terms have sublinear growth if $N\geq 5$. However, none of these solutions are pinwheel solutions.

Pinwheel solutions were first found by Wei and Weth in \cite{ww} who studied a subcritical Schrödinger system of two equations in dimensions $N=2,3$. Pinwheel solutions for subcritical Schrödinger systems have been also exhibited, for instance, in \cite{cp, pv, pw}.

The proof of Theorem \ref{thm:existence} is based on a concentration compactness argument (see Theorem \ref{thm:minimizing}) that describes the behavior of minimizing pinwheel sequences for the system in the unit ball $B_1$ in $\rn$. It yields the existence of a least energy pinwheel solution, either in $B_1$, or in a half space, or in all of $\rn$. The first two cases are discarded by means of a unique continuation principle for systems proved in \cite{chs}.

\begin{questions}
Some interesting questions remain open. For instance,
\begin{enumerate}
\item[$(1)$] If $\bf w_{\infty}=(w_{1,\infty},\dots, w_{\ell,\infty})$ is the function given by \emph{Theorem} \ref{thm:main2}, is it true that $\sum_{i=1}^\ell (-1)^{i+1}w_{i,\infty}$ is a sign-changing solution of the Yamabe problem \eqref{eq:yamabe0}?
\item[$(2)$] Does the sign-changing solution $\widehat{w}_\ell$ of \eqref{eq:yamabe0} given by \emph{Theorem} \ref{thm:nodal} have precisely $\ell$ nodal domains? And, if so, do the nodal domains form an $\ell$-pinwheel partition for \eqref{eq:yamabe0}?
\end{enumerate}
\end{questions}

The paper is organized as follows. In Section \ref{sec:variational_setting} we describe our variational symmetric setting. In Section \ref{sec:existence} we study the behavior of minimizing sequences in the unit ball and use it to prove Theorem \ref{thm:existence}. Sections \ref{sec:optimal_partition} and \ref{sec:nodal} are devoted to the proofs of Theorems \ref{thm:main2} and \ref{thm:nodal} respectively.

\section{The variational setting}
\label{sec:variational_setting}

Let $\o$ be a domain in $\rn$, $N\geq 4$, $\beta<0$ and $p:=\frac{N}{N-2}$, and consider the system
\begin{equation} \label{eq:system}
\begin{cases}
-\Delta u_i = |u_i|^{2p-2}u_i+\sum\limits_{\substack{j=1\\j\neq i}}^\ell\beta|u_j|^p|u_i|^{p-2}u_i, \\
u_i\in D_0^{1,2}(\o),\qquad i=1,\ldots,\ell.
\end{cases}
\end{equation}
 Let
$$\langle u,v\rangle:=\io\nabla u\cdot\nabla v\qquad\text{and}\qquad\|u\|:=\sqrt{\langle u,u\rangle}$$
be the inner product and the norm of $u,v\in D_0^{1,2}(\o)$, and $|u|_{2p}$ be the norm of $u$ in $L^{2p}(\o)$. The solutions $\bf u=(u_1,\ldots,u_\ell)$ to \eqref{eq:system} are the critical points of the functional $\cJ:(D_0^{1,2}(\o))^\ell\to\r$ given by 
$$\cJ(\bf u):= \frac{1}{2}\sum_{i=1}^\ell\|u_i\|^2 - \frac{1}{2p}\sum_{i=1}^\ell |u_i|_{2p}^{2p}-\frac{\beta}{2p}\sum_{\substack{i,j=1 \\ i\neq j}}^\ell\io |u_i|^p|u_j|^p,$$
which is of class $\cC^1$. Its $i\text{-th}$ partial derivative is
$$\partial_i\cJ(\bf u)v=\langle u_i,v\rangle-\io|u_i|^{2p-2}u_iv-\beta\sum\limits_{\substack{j=1\\j\neq i}}^\ell\io|u_j|^p|u_i|^{p-2}u_iv$$
for any $\bf u\in(D_0^{1,2}(\o))^\ell, \ v\in D_0^{1,2}(\o)$.

We are interested in pinwheel solutions to \eqref{eq:system} for suitable domains $\o$. The variational setting that produces such solutions is obtained through the successive application of the principle of symmetric criticality to the symmetries described below.

\subsection{The action of the complex unit circle}

Let $G$ be the set of complex numbers on the unit circle $\s1$ acting on $\rn\equiv\cc\times\cc\times\r^{N-4}$ as follows:
\begin{equation}\label{eq:s1action}
gx:=(gz_1,\overline{g}z_2,y)\quad\text{for every \ }g \in\s1\text{ \ and \ }x=(z_1,z_2,y)\in\cc\times\cc\times\r^{N-4},	
\end{equation}
where $\overline{g}$ is the complex conjugate of $g$. Then $G$ is a subgroup of $O(N)$.

Let $\o$ be $G$-invariant, that is, that $gx\in\o$ for every $g\in G$ and $x\in\o$. If $g\in G$ and $u\in D_0^{1,2}(\o)$ we define $gu\in D_0^{1,2}(\o)$ by $gu(x):=u(g^{-1}x)$. This gives an isometric action of $G$ on $(D_0^{1,2}(\o))^\ell$ by taking
$$g\bf u:=(gu_1,\ldots,gu_\ell),\quad\text{if \ } g\in G\text{ \ and \ }\bf u=(u_1,\ldots,u_\ell)\in (D_0^{1,2}(\o))^\ell.$$
The $G$-fixed point space of $(D_0^{1,2}(\o))^\ell$ is the subspace
\begin{align*}
\cD(\o):&=\{\bf u\in (D_0^{1,2}(\o))^\ell:g\bf u=\bf u \text{ for all }g\in G\}\\
&=\{\bf u\in (D_0^{1,2}(\o))^\ell:u_j\text{ is }G\text{-invariant for every }j=1,\ldots,\ell\}.
\end{align*}
The functional $\cJ$ is $G$-invariant. So, by the principle of symmetric criticality \cite[Theorem 1.28]{w}, the critical points of the restriction of $\cJ$ to $\cD(\o)$ are the solutions to the system \eqref{eq:system} satisfying $(S_1)$ for this group $G$.

\subsection{The action of the cyclic group}

Let $\tau:\cc^2\to\cc^2$ be given by $\tau(z_1,z_2):=(-\overline{z}_2,\overline{z}_1)$ and, for each $j\in\n$, define $\vr_\ell^j\in O(N)$ by
\begin{equation}\label{eq:eq:vrj}
\vr_\ell^jx:=\left(\Big(\cos\frac{\pi j}{\ell}\Big)z+\Big(\sin\frac{\pi j}{\ell}\Big)\tau z,\,y\right)\qquad\text{for every \ }x=(z,y)\in\cc^2\times\r^{N-4}.
\end{equation}
Since $\tau z$ is orthogonal to $z$ and $|\tau z|=|z|$, we have that $\vr_\ell^j\in O(N)$ and $\vr_\ell^k\vr_\ell^j=\vr_\ell^{k+j}$. Note also that $\vr_\ell^jg=g\vr_\ell^j$ for every $g\in G$ and every $j\in\n$. 

Set $\vr_\ell:=\vr_\ell^1$ and let $\Gamma_\ell$ be the subgroup of $O(N)$ generated by $G\cup\{\vr_\ell\}$. \emph{In the remainder of this section we will assume that $\o$ is $\Gamma_\ell$-invariant}.

We denote by $\sigma^j:\{1,\ldots,\ell\}\to\{1,\ldots,\ell\}$ the permutation $\sigma^j(r):=r+j \mod\ell$, \ $j=0,\ldots,\ell-1$.

\begin{proposition}
Let $\zl=\{0,\ldots,\ell-1\}$ be the additive group of integers modulo $\ell$.
\begin{itemize}
\item[$(a)$] For each $j\in\zl$ the function $\vr_\ell^j:\cD(\o)\to\cD(\o)$ given by
$$\vr_\ell^j\bf u(x):=(u_{\sigma^j(1)}(\vr_\ell^{-j}x),\ldots,u_{\sigma^j(\ell)}(\vr_\ell^{-j}x)),\quad\text{where \ }\bf u=(u_1,\ldots,u_\ell),$$
is well defined and it is a linear isometry.
\item[$(b)$] $j\mapsto\vr_\ell^j$ is a well defined action of $\zl$ on $\cD(\o)$.
\end{itemize}
\end{proposition}

\begin{proof}
$(a):$ \ Let $\bf u\in\cD(\o)$ and $g\in G$. Since $u_j$ satisfies $(S_1)$ and $\vr_\ell^jg=g\vr_\ell^j$ we have
\begin{align*}
g[\vr_\ell^j\bf u](x)&=[\vr_\ell^j\bf u](g^{-1}x)\\
&=(u_{\sigma^j(1)}(\vr_\ell^{-j}g^{-1}x),\ldots,u_{\sigma^j(\ell)}(\vr_\ell^{-j}g^{-1}x))\\
&=(u_{\sigma^j(1)}(g^{-1}\vr_\ell^{-j}x),\ldots,u_{\sigma^j(\ell)}(g^{-1}\vr_\ell^{-j}x))\\
&=(u_{\sigma^j(1)}(\vr_\ell^{-j}x),\ldots,u_{\sigma^j(\ell)}(\vr_\ell^{-j}x))=\vr_\ell^j\bf u(x). 
\end{align*}
This shows that $\vr_\ell^j\bf u\in\cD(\o)$. The function $\vr_\ell^j:\cD(\o)\to\cD(\o)$ is clearly linear and bijective, and it satisfies $\|\vr_\ell^j\bf u\|=\|\bf u\|$.

$(b):$ \ $\vr_\ell^{-\ell}(z,y)=(-z,y)$ for every $(z,y)\in\cc^2\times\r^{N-4}$. So, if $\bf u\in\cD(\o)$, then
\begin{align*}
\vr_\ell^\ell\bf u(z,y)&=(u_{\sigma^\ell(1)}(\vr_\ell^{-\ell}( z,y)),\ldots,u_{\sigma^\ell(\ell)}(\vr_\ell^{-\ell}(z,y)))\\
&=(u_1(-z,y),\ldots,u_\ell(-z,y))=\bf u(-z,y)=\bf u(z,y).
\end{align*}
This shows that $\vr_\ell^{\ell}:\cD(\o)\to\cD(\o)$ is the identity. So $j\mapsto\vr_\ell^j$ is a well defined homomorphism from $\zl$ into the group of linear isometries of $\cD(\o)$.
\end{proof}

\begin{remark} \label{rem:1}
\emph{The $G$-orbit of $z=(z_1,z_2)\in\cc^2$ with respect to the action defined in \eqref{eq:s1action}, which, by definition, is the set
$$Gz:=\{(g z_1,\overline{g}z_2):g\in G\},$$
is the circle of radius $|z|$ contained in the plane $\{\alpha z + \beta\i z:\alpha,\beta\in\r\}$, where $\i$ is the imaginary unit. If $z\neq 0$, then
$$\tau z\in Gz \ \Longleftrightarrow \ \text{either \ }\i z_1=-\overline{z}_2\text{ \ or \ }\i z_1=\overline{z}_2.$$
So the $G$-orbits of $(z_1,\i\overline z_1)$ and $(z_1,-\i\overline z_1)$ are the only fixed points of the action of $\zl$ on the $G$-orbit space of $\cc^2\smallsetminus\{0\}$ given by $j\mapsto\vr_\ell^j$. We stress that there is no free action of $\zl$ on the $G$-orbit space of $\cc^2\smallsetminus\{0\}$ for $\ell>2$.}
\end{remark}
\medskip

The $\zl$-fixed point space of $\cD(\o)$ is the space
\begin{align*}
\mathscr{D}(\o):&=\cD(\o)^{\zl}=\{\bf u\in\cD(\o):\vr_\ell^j\bf u=\bf u\text{ \ for all \ }j\in\zl \} \\
&=\{\bf u\in (D_0^{1,2}(\o))^\ell:u_j\text{ is }G\text{-invariant and }u_{j+1}=u_j\circ\vr_\ell\text{ for every }j=1,\ldots,\ell\}.
\end{align*}
The functional $\cJ|_{\cD(\o)}:\cD(\o)\to\r$ is $\zl$-invariant. So, by the principle of symmetric criticality, the critical points of its restriction to $\mathscr{D}(\o)$ are the pinwheel solutions of the system \eqref{eq:system}. Abusing notation, we  write
\begin{equation}\label{eq:energy}
\cJ:=\cJ|_{\mathscr{D}(\o)}:\mathscr{D}(\o)\to\r.
\end{equation}
Note that
\begin{align*}
\cJ'(\bf u)\bf v=\sum_{i=1}^\ell\partial_i\cJ(\bf u)v_i=\ell\,\partial_j\cJ(\bf u)v_j\quad\text{for any \ }\bf u,\bf v\in\mathscr{D}(\o)\text{ and }j=1,\ldots,\ell.
\end{align*}
By condition $(S_2)$, if $\bf u\in\mathscr{D}(\o)$ and $\bf u\neq 0$, then every component of $\bf u$ is nontrivial. So the fully nontrivial critical points of $\cJ|_{\mathscr{D}(\o)}$ belong to the Nehari manifold
\begin{align}\label{eq:nehari}
\cN(\o)=\{\bf u\in\mathscr{D}(\o):\bf u\neq 0, \ \cJ'(\bf u)\bf u=0\}.
\end{align}
Set
$$c(\o):=\inf_{\bf u\in\cN(\o)}\cJ(\bf u).$$

\begin{lemma}\label{lem:nehari}
\begin{itemize}
	\item[$(i)$] $\cN(\o)\neq\emptyset$.
	\item[$(ii)$] If $\bf u=(u_{1},\dots,u_{\ell})\in\cN(\o)$, then 
	$$S^{N/2}\leq\|u_i\|^2\leq|u_i|_{2p}^{2p}\qquad\text{for every \ }i=1,\ldots,\ell,$$
	and $c(\o)\geq \frac{\ell}{N} S^{\frac{N}{2}}$, where $S$ is the best constant for the Sobolev embedding $D^{1,2}(\mathbb{R}^{N})\hookrightarrow L^{2p}(\mathbb{R}^{N})$. 
	\item[$(iii)$] $\cN(\o)$ is a closed $\cC^1$-submanifold of codimension $1$ of $\mathscr{D}(\o)$ and a natural constraint for $\cJ|_{\mathscr{D}(\o)}$.
\end{itemize}
\end{lemma}

\begin{proof}
$(i):$ \ Let $x_0=(z_0,y_0)\in\o$ be such that $(\tau z_0,y_0)\notin Gx_0$. Since $\o$ is open, such a point always exists, see Remark \ref{rem:1}. Then $\vr_\ell^jx_0\notin Gx_0$ if $j=1,\ldots,\ell-1$. Hence, the $G$-orbits of $x_j:=\vr_\ell^jx_0$ and $x_0$ are disjoint. It follows that the $G$-orbits of the points $x_j$ with  $j=0,\ldots,\ell-1$ are pairwise disjoint. Fix $\eps>0$ such that $2\eps<\dist(Gx_j,Gx_i)$ for every pair $i\neq j$ and set $U_j:=\{x\in\rn:\dist(x,Gx_j)<\eps\}$. Then $U_j\subset\o$ and $U_j$ is $G$-invariant. Choose a nontrivial $G$-invariant function $u_1\in\cC_c^\infty(U_0)$ such that $\|u_1\|^2=|u_1|_{2p}^{2p}$ and, for $j=1,\ldots,\ell-1$, define $u_{j+1}(x):=u_1(\vr_\ell^{-j}x)$. Then, $u_{j+1}\in\cC_c^\infty(U_j)$, $\bf u=(u_1,\ldots,u_\ell)\in\mathscr{D}(\o)$ and, since any two components have disjoint supports, $\bf u\in\cN(\o)$. 

$(ii):$ \ If $\bf u=(u_{1},\dots,u_{\ell})\in  \cN(\o)$ then $\cJ'(\bf u)\bf u=0$ and, since $\beta<0$,
\begin{equation*}
\|u_{i}\|^{2}=\irn|u_i|^{2p}+\beta\sum\limits_{\substack{j=1\\j\neq i}}^\ell\irn|u_j|^p|u_i|^{p}\leq |u_i|_{2p}^{2p}.
\end{equation*}
Hence,
\begin{equation*}
	0<S\leq \frac{\|u_{i}\|^{2}}{|u_i|_{2p}^2}\leq\left( \|u_{i}\|^{2}\right)^{\frac{2p-2}{2p}} 
\end{equation*}
and
\begin{equation*}
\cJ(\bf u)=\frac{1}{N}\|\bf u\|^{2}\geq \frac{\ell}{N} S^{\frac{N}{2}}.
\end{equation*}
The statement $(ii)$ follows from these inequalities.

$(iii):$ \ A standard argument shows that $\cN(\o)$ is a closed $\cC^1$-submanifold of codimension $1$ of $\mathscr{D}(\o)$ and a natural constraint for $\cJ|_{\mathscr{D}(\o)}$. 
\end{proof}

\begin{definition}\label{def:le}
A function $\bf u\in\cN(\o)$ such that $\cJ(\bf u)=c(\o)$ is called a \textbf{least energy pinwheel solution} to the system \eqref{eq:system}.
\end{definition}

The following result is a special case of \cite[Theorem 1.5]{chs}. For completeness, we give the details.

\begin{proposition} \label{prop:no_minimum}
\begin{itemize}
\item[$(i)$]If $\o\cap(\{0\}\times\r^{N-4})\neq\emptyset$, then  $c(\o)=c(\rn)$.
\item[$(ii)$]If, in addition, $\rn\smallsetminus\o$ has nonempty interior, then the system \eqref{eq:system} does not have a least energy pinwheel solution.
\end{itemize}
\end{proposition}

\begin{proof}
$(i):$ \ As $\cN(\o)\subset\cN(\rn)$ via trivial extension, we have that $c(\o)\geq c(\rn)$. To prove the opposite inequality note that, as the space of $G$-invariant functions in $\cC^\infty_c(\rn)$ is dense in the space of $G$-invariant functions in $D^{1,2}_0(\rn)$ there exists a sequence $\bf\vp_k=(\vp_{1,k},\ldots,\vp_{\ell,k})\in\cN(\rn)$ such that $\vp_{j,k}\in\cC^\infty_c(\rn)$ and $\cJ(\bf\vp_k)\to c(\rn)$. Let $\xi\in\o\cap(\{0\}\times\r^{N-4})$ and $r>0$ be such that $B_r(\xi)\subset\o$, where $B_r(\xi)$ is the ball centered at $\xi$ of radius $r$.  Fix $\eps_k>0$ such that $\eps_kz\in B_r(0)$ for every $z\in\supp(\vp_{1,k})$. Since $\bf\vp_k$ satisfies $(S_2)$ we have that $\eps_kx\in B_r(0)$ for every $x\in\supp(\vp_{j,k})$ and $j=1,\ldots,\ell$. Define
$$\widetilde{\vp}_{j,k}(x):=\eps_k^\frac{2-N}{2}\vp_{j,k}\Big(\frac{x-\xi}{\eps_k}\Big).$$
Then $\widetilde{\vp}_{j,k}\in\cC^\infty_c(\o)$ and, since $g\xi=\xi$ for every $g\in G$ and $\vr_\ell\xi=\xi$, we have that $\widetilde{\bf\vp}_k=(\widetilde{\vp}_{1,k},\ldots,\widetilde{\vp}_{\ell,k})\in\mathscr{D}(\o)$. Furthermore, as
$$\|\widetilde{\vp}_{j,k}\|^2=\|\vp_{j,k}\|^2\qquad\text{and}\qquad\int_{B_1}|\widetilde{\vp}_{j,k}|^p|\widetilde{\vp}_{i,k}|^p=\irn|\vp_{j,k}|^p|\vp_{i,k}|^p,$$
we have that $\widetilde{\bf\vp}_k\in\cN(\o)$ and $\cJ(\widetilde{\bf\vp}_k)=\cJ(\bf\vp_k)\to c(\rn)$. This shows that $c(\o)\leq c(\rn)$.

$(ii):$ Arguing by contradiction, assume that some $u\in\cN(\o)$ satisfies $\cJ(u)=c(\o)$. Then, its trivial extension $\bar u$ to $\rn$ belongs to $\cN(\rn)$ and satisfies $\cJ(\bar u)=c(\rn)$. Thus,
$\bar u$ is a nontrivial solution of the system \eqref{eq:system_rn} all of whose components vanish in $\rn\smallsetminus\o$. This contradicts the unique continuation principle for systems \cite[Theorem 1.1]{chs}.
\end{proof}

\section{A pinwheel solution of the Yamabe system}
\label{sec:existence}

\begin{lemma}\label{lem:G}
Let $G$ act on $\rn\equiv\cc^2\times\r^{N-4}$ as in \eqref{eq:G}. Then any given sequences $(\eps_k)$ in $(0,\infty)$ and $(\zeta_{k})$ in $\rn$ contain subsequences that satisfy one of the following statements:
\begin{enumerate}
\item[$(i)$] either there exist $\eta_k\in\{0\}\times\r^{N-4}$ and $C_1>0$ such that $\eps_{k}^{-1}|\zeta_{k}-\eta_k|<C_1$ for all $k\in\n$,
\item[$(ii)$] or, for each $m\in\n$, there exist $g_1,\ldots,g_m\in G$ such that $\eps_{k}^{-1}|g_i\zeta_{k}-g_j\zeta_{k}|\rightarrow\infty$ for any $i\neq j$.
\end{enumerate}
\end{lemma}

\begin{proof}
After passing to a subsequence, we have that $\eps_k^{-1}\dist(\zeta_k,\{0\}\times\r^{N-4})\to a\in[0,\infty]$. If $a<\infty$ then $(i)$ holds true. If $a=\infty$ we may assume that $\zeta_k=(z_k,y_k)\in\cc^2\times\r^{N-4}$ with $z_k\neq 0$ for all $k\in\n$, and passing to a subsequence we have that
$$\frac{z_k}{|z_k|}\to z\quad\text{in \ }\cc^2.$$
As $z\neq 0$, for each $m\in\n$ there exist $g_1,\ldots,g_m\in G$ such that $g_i(z,0)\neq g_j(z,0)$ if $i\neq j$. Let $3d:=\min_{i\neq j}|g_i(z,0)-g_j(z,0)|$, and let $k_0\in\n$ be such that
$$\Big|\frac{z_k}{|z_k|}-z\Big|<d\qquad\text{if \ }k\geq k_0.$$
Then, since the $g_i$'s are a linear isometries, we obtain
\begin{align*}
d\leq \Big|\frac{g_i(z_k,0)}{|z_k|}-\frac{g_j(z_k,0)}{|z_k|}\Big|=\Big|\frac{g_i\zeta_k}{|z_k|}-\frac{g_j\zeta_k}{|z_k|}\Big|\qquad\text{if \ }i\neq j\text{ \ and \ }k\geq k_0. 
\end{align*}
Therefore, $\eps_k^{-1}|z_k|\,d\leq \eps_k^{-1}|g_i\zeta_k-g_j\zeta_k|$ and, since $\eps_k^{-1}|z_k|=\eps_k^{-1}\dist(\zeta_k,\{0\}\times\r^{N-4})\to\infty$, we get that 
$$\eps_k^{-1}|g_i\zeta_k-g_j\zeta_k|\to\infty\qquad\text{if \ }i\neq j.$$
This shows that, if $a=\infty$, then $(ii)$ holds true.
\end{proof}

The following theorem describes the behavior of minimizing sequences to the system in the unit ball.

\begin{theorem} \label{thm:minimizing} 
Let $B_1$ be the unit ball centered at the origin in $\rn$ and $(\bf u_k)$ be a sequence in $\cN(B_1)$ such that $\cJ(\bf u_k)\to c(B_1)$. Then, after passing to a subsequence, there exist sequences $(\eps_k)$ in $(0,\infty)$ and $(\xi_k)$ in $B_1\cap(\{0\}\times\r^{N-4})$ and a least energy pinwheel solution $\bf w=(w_1,\ldots,w_\ell)$ to the system \eqref{eq:system_rn} in $\rn$, such that
\begin{align*}
\eps_k^{-1}\dist(\xi_k,\partial B_1)\to\infty,\qquad\text{and}\qquad\lim_{k\to\infty}\|\bf u_k - \widehat{\bf w}_k\|=0,
\end{align*}
where $\widehat{\bf w}_k=(\widehat{w}_{1,k},\ldots, \widehat{w}_{\ell,k})$ is given by
\begin{equation*}
 \widehat{w}_{i,k}(x) =\eps_{k}^{\frac{2-N}{2}}w_i\Big(\frac{x-\xi_k}{\eps_{k}}\Big)\qquad \text{for every \ } k\in\n \text{ \ and \ }i=1,\ldots,\ell.
\end{equation*}
\end{theorem}

\begin{proof}
Using Ekeland's variational principle and the principle of symmetric criticality we may assume that $\cJ'(\bf u_{k})\rightarrow 0 $  in $[(D^{1,2}_0(B_1))^\ell]'$. Since $\cJ(\bf u_{k})=\frac{1}{N}\|\bf u_{k}\|^{2}\to c(B_1)$, the sequence $(\bf u_{k})$ is bounded in $(D^{1,2}_0(B_1))^\ell$ and therefore, after passing to a subsequence, we have that $\bf u_{k}\rh\bf u$ weakly in $\mathscr{D}(B_1)$. A standard argument shows that $\bf u$ is a solution of the system \eqref{eq:system} in $B_1$. So, if $\bf u\neq 0$, then $\bf u\in \cN(B_1)$ and
$$c(B_1)\leq\cJ(\bf u)=\frac{1}{N}\|\bf u\|^2\leq\liminf_{k\to\infty}\frac{1}{N}\|\bf u_k\|^2=\liminf_{k\to\infty}\cJ(\bf u_k)=c(B_1).$$
This shows that $\bf u$ is a least energy solution to the system \eqref{eq:system} in $B_1$, contradicting Proposition \ref{prop:no_minimum}$(ii)$.
 
As a consequence, $\bf u_{k}\rh\bf 0$ weakly in $(D^{1,2}_0(B_1))^\ell$. Now, fix $\delta\in\mathbb{R}$ such that $0<2\delta<S^{N/2}$. Since Lemma \ref{lem:nehari} states that $S^{N/2}\leq |u_{1,k}|_{2p}^{2p}$ for each $k$, there exist bounded sequences $(\eps_{k})$ in $(0,\infty)$ and $(\zeta_{k})$ in $\mathbb{R}^{N}$ such that, after passing to a subsequence,
\begin{equation}\label{desi}
\delta=\sup_{x\in\mathbb{R}^{N}}\int_{B_{\eps_{k}}(x)}|u_{1,k}|^{2p}= \int_{B_{\eps_{k}}(\zeta_{k})}|u_{1,k}|^{2p}
\end{equation}
Applying Lemma \ref{lem:G} to the sequences $(\eps_{k})$ in $(0,\infty)$ and $(\zeta_{k})$, we have two possible options. Let us show that option $(ii)$ is impossible.

By contradiction, assume that for any $m\in\mathbb{N}$ there exist $g_{1},\dots,g_{m}\in G$ such that $\eps_{k}^{-1}\left|g_{i}\zeta_{k}-g_{j}\zeta_{k} \right|\rightarrow\infty $ for every $i\neq j$. Thus, for sufficiently large $k$ we have
\begin{equation*}
	B_{\eps_{k}}(g_{i}\zeta_{k})\cap 	B_{\eps_{k}}(g_{j}\zeta_{k})=\emptyset \ \ \mbox{ if } i\neq j
\end{equation*}
and, since $u_{1,k}$ is $G$-invariant, from \eqref{desi} we get
\begin{equation*}
m\delta=\sum_{i=1}^m\int_{B_{\eps_{k}}(g_{i}\zeta_{k})}|u_{1,k}|^{2p}\leq \int_{B_1}|u_{1,k}|^{2p}.
\end{equation*} 
This is a contradiction because $(u_{1,k})$ is bounded in $D^{1,2}_0(B_1)$.  

As a result, option $(i)$ must hold true, i.e., after passing to a subsequence, there exist $\eta_k\in\{0\}\times\r^{N-4}$ and $C_1>0$ such that $\eps_{k}^{-1}|\zeta_{k}-\eta_k|<C_1$ for all $k\in\n$. Then, $B_{\eps_k}(\zeta_k)\subset B_{(C_1+1)\eps_k}(\eta_k)$ and, as a consequence,
\begin{equation} \label{eq:eta}
0<\delta=\int_{B_{\eps_k}(\zeta_k)}|u_{1,k}|^{2p}\leq \int_{B_{(C_1+1)\eps_k}(\eta_k)}|u_{1,k}|^{2p}.
\end{equation}
Passing to a subsequence we have that
\begin{equation*}
d_{k}:=\varepsilon_k^{-1}\dist(\eta_k,\partial B_1)\to d\in [0,\infty].
\end{equation*} 
If $d\in[0,\infty)$, for each $k$ we take $\xi_k\in\partial B_1\cap(\{0\}\times\r^{N-4})$ such that $|\eta_k-\xi_k|=\dist(\eta_k,\partial B_1)$. Whereas, if $d=\infty$, we set $\xi_k:=\eta_k$. In this last case \eqref{eq:eta} implies that $\dist(\xi_k,B_1)\leq (C_1+1)\eps_k$ and, as $d=\infty$, necessarily $\xi_k\in B_1$. Summing up, there are two possibilities:
\begin{itemize}
\item[$(1)$] either $\xi_k\in\partial B_1\cap(\{0\}\times\r^{N-4})$,
\item[$(2)$] or $\xi_k\in B_1\cap(\{0\}\times\r^{N-4})$ and $\varepsilon_k^{-1}\mathrm{dist}(\xi_k,\partial B_1)\to\infty$.
\end{itemize}
Furthermore, in both cases there exists $C_0>0$ such that $\eps_k^{-1}|\zeta_k-\xi_k|\leq C_0$. Thus, $B_{\varepsilon_k}(\zeta_k)\subset B_{(C_0+1)\varepsilon_k}(\xi_k)$ for large enough $k$ and, as a consequence,
\begin{equation} \label{eq:xi}
0<\delta =\sup_{x\in\rn}\int_{B_{\eps_k}(x)}|u_{1,k}|^{2p} = \int_{B_{\varepsilon_k}(\zeta_k)}|u_{1,k}|^{2p}\leq \int_{B_{(C_0+1)\varepsilon_k}(\xi_k)}|u_{1,k}|^{2p}.  
\end{equation}

Let $\Omega_{k}:=\{x\in\mathbb{R}^{N}: \eps_{k}x+\xi_k\in B_1 \}$ and consider the function $\bf w_{k}=(w_{1,k},\ldots,w_{\ell,k})$ defined by
\begin{equation*}
w_{i,k}(x):=\eps_{k}^{\frac{N-2}{2}}u_{i,k}(\eps_{k}x+\xi_k)\quad\text{if \ }x\in\o_k,\qquad w_{i,k}(x):=0\quad\text{otherwise}.
\end{equation*}
Since $\xi_k\in\{0\}\times\r^{N-4}$ we have that $g\xi_k=\xi_k$ for every $g\in G$ and $\vr_\ell\xi_k=\xi_k$. Therefore, $\bf w_{k}\in\mathscr{D}(\rn)$ for every $k$ and, as $\|w_{i,k}\|=\|u_{i,k}\|$, after passing to a subsequence we get that
\begin{equation*}
\bf w_{k}\rightharpoonup \bf w \mbox{ weakly in } \mathscr{D}(\mathbb{R}^{N}),\qquad \bf w_{k}\rightarrow \bf w \mbox{ in }  (L^{2}_{\mathrm{loc}}(\mathbb{R}^{N}))^{\ell} \qquad \mbox{ and }\qquad \bf w_{k}\rightarrow \bf w  \mbox{ \ a.e. in } ( \mathbb{R}^{N}) ^{\ell}. 
\end{equation*}
In addition, equation \eqref{eq:xi} yields
\begin{equation}\label{des1}
	\delta=\sup_{x\in\mathbb{R}^{N}} \int_{B_{1}(x)}|w_{1,k}|^{2p}\leq 	\int_{B_{(C_{0}+1)}(0)}|w_{1,k}|^{2p}.
\end{equation}
Next we show that $\bf w\neq\bf 0$. Let $\phi \in\cC_c^\infty(\rn)$ and set $\phi_{k}(x):=\phi\big(\frac{x-\xi_k}{\eps_k}\big)$. Then, $(\phi_{k}^{2}u_{1,k})$ is a bounded sequence in $D_{0}^{1,2}(B_1)$ and, since $\cJ'(\bf u_{k})\rightarrow 0 $  in $[(D^{1,2}_0(B_1))^\ell]'$, performing a change of variable we get
\begin{equation*}
o(1)=\partial_1\cJ(\bf u_{k})[\phi_{k}^{2}u_{1,k}]=\int_{\mathbb{R}^{N}}\nabla w_{1,k}\cdot\nabla(\phi^{2}w_{1,k})-\irn|w_{1,k}|^{2p}\phi^{2}-\beta\sum\limits_{\substack{j=1\\j\neq 1}}^\ell\irn|w_{j,k}|^p|w_{1,k}|^{p}\phi^{2}.
\end{equation*}
Assume, by contradiction, that $\bf w=\bf 0$, then $w_{1,k}\rightarrow 0 \mbox{ in } L^{2}_{\mathrm{loc}}(\mathbb{R}^{N})$. As a consequence, if $\phi \in\cC_c^\infty(B_1(x))$ for some $x\in\rn$ we get
\begin{align*}
		\irn|\nabla (\phi w_{1,k})|^{2}&= \irn\nabla w_{1,k}\cdot\nabla(\phi_{k}^{2}w_{1,k})+\irn w_{1,k}^{2}|\nabla\phi|^{2}\\
		&=\int_{B_{1}(z)}|w_{1,k}|^{2p}\phi^{2}+\beta\sum\limits_{\substack{j=1\\j\neq 1}}^\ell\irn|w_{j,k}|^p|w_{1,k}|^{p}\phi^{2}+o(1)\\
		&\leq \int_{B_{1}(z)}|w_{1,k}|^{2p}\phi^{2}+o(1)\\
		&\leq \Big( \int_{B_{1}(z)}|w_{1,k}|^{2p}\Big)^{\frac{2p-2}{2p}} \Big( \irn|\phi w_{1,k}|^{2p}\Big)^{\frac{2}{2p}} +o(1)\\ 
		&\leq  \delta^{2/N}S^{-1}\Big(\irn|\nabla (\phi w_{1,k})|^{2}\Big)+o(1)\\
		&< \Big( \frac{1}{2}\Big) ^{2/N}\Big(\irn|\nabla (\phi w_{1,k})|^{2}\Big)+o(1),
\end{align*}
because $\beta<0$, $2\delta<S^{N/2}$ and $\delta$ satisfies \eqref{des1}. Hence, $\irn|\nabla (\phi w_{1,k})|^{2}=o(1)$ and therefore $|\phi w_{1,k}|_{2p}=o(1)$ for every $\phi \in\cC_c^\infty(B_{1}(x))$ and any $x\in\mathbb{R}^{N}$. This implies that $w_{1,k}\rightarrow 0$ in $L_{\mathrm{loc}}^{2p}(\mathbb{R}^{N})$, contradicting \eqref{des1}. This proves that $\bf w=(w_1,\ldots,w_\ell)\neq \bf 0$. 

Note that, since $u_{1,k}\rh 0$ and $w_{1,k}\rh w_1\neq 0$ weakly in $D^{1,2}(\rn)$, we have that $\eps_k\to 0$.

Let us now analyze the alternatives $(1)$ and $(2)$. If $(1)$ holds true, passing to a subsequence we have that $\xi_k\to\xi\in \partial B_1\cap(\{0\}\times\r^{N-4})$. Then $\mathbb{H}:=\{x\in\rn:\xi\cdot x<0\}$ is $\Gamma_\ell$-invariant (where $\Gamma_\ell$ is the group generated by $G\cup\{\vr_\ell)$), and one can easily see that $\bf w\in\mathscr{D}(\mathbb{H})$  and that $\supp(\vp)\subset \Omega_{k}$ for large enough $k$ if $\vp\in\cC^\infty_c(\mathbb{H})$. Therefore the support of the function $\vp_k(x):=\eps_{k}^{(2-N)/2}\vp\left(\frac{x-\xi_k}{\eps_k}\right)$ is contained in $B_1$ and, since the sequence $(\vp_k)$ is bounded in $D^{1,2}_0(B_1)$, performing a change of variable we obtain
\begin{align*}
\partial_i\cJ(\bf w_{k})\vp &=\irn \nabla(w_{i,k})\cdot \nabla\vp-\irn|w_{i,k}|^{2p-2}w_{i,k}\vp-\beta\sum\limits_{\substack{j=1\\j\neq i}}^\ell\irn|w_{j,k}|^p|w_{i,k}|^{p-2}w_{i,k}\vp\\
&=\partial_i\cJ(\bf u_{k})\vp_k=o(1).
\end{align*}
As $\bf w_{k}\rightharpoonup \bf w$ weakly in $\mathscr{D}(\mathbb{R}^{N})$, we derive that
\begin{equation*}
\partial_i\cJ( \bf w)\vp=	\lim_{k\to\infty}\partial_i\cJ( \bf w_{k})\vp=0\qquad \text{for every \ }\vp\in\cC^\infty_c(\mathbb{H}),\quad i=1,\ldots, \ell.
\end{equation*}
Therefore $\bf w\in\cN(\mathbb{H})$, and using Proposition \ref{prop:no_minimum}$(i)$ we obtain
$$c(\mathbb{H})\leq \cJ(\bf w)=\frac{1}{N}\|\bf w\|^2\leq \liminf_{k\to\infty}\frac{1}{N} \|\bf w_k\|^2=\liminf_{k\to\infty}\frac{1}{N} \|\bf u_k\|^2=c(B_1)=c(\mathbb{H}).$$
This shows that $\bf w$ is a least energy pinwheel solution of \eqref{eq:system} in $\mathbb{H}$, contradicting Proposition \ref{prop:no_minimum}$(ii)$.

So, we are left with alternative $(2)$, i.e., $\xi_k\in B_1\cap(\{0\}\times\r^{N-4})$ and $\eps_k^{-1}\dist(\xi_k,\partial B_1)\to\infty$. It follows that $\supp(\vp)\subset \Omega_{k}$ for large enough $k$ for any $\vp\in\cC^\infty_c(\rn)$, and arguing as before we see that
\begin{equation*}
\partial_i\cJ(\bf w)\vp=\lim_{k\to\infty}\partial_i\cJ( \bf w_{k})\vp=0\qquad \text{for every \ }\vp\in\cC^\infty_c(\rn),\quad  i=1,\ldots, \ell.
\end{equation*}
Thus $\bf w\in\cN(\rn)$, and using Proposition \ref{prop:no_minimum}$(i)$ we obtain
$$c(\rn)\leq \cJ(\bf w)=\frac{1}{N}\|\bf w\|^2\leq \liminf_{k\to\infty}\frac{1}{N} \|\bf w_k\|^2=\liminf_{k\to\infty}\frac{1}{N} \|\bf u_k\|^2=c(B_1)=c(\rn).$$
This shows that $\bf w$ is a least energy pinwheel solution of \eqref{eq:system_rn} and that $\bf w_k\to\bf w$ strongly in $\mathscr{D}(\rn)$.
So, setting $\widehat{\bf w}_k=(\widehat{w}_{1,k},\ldots, \widehat{w}_{\ell,k})$ with
\begin{equation*}
 \widehat{w}_{i,k}(x) :=\eps_{k}^{\frac{2-N}{2}}w_i\Big(\frac{x-\xi_k}{\eps_k}\Big),
\end{equation*}
we get that $\|\bf u_k - \widehat{\bf w}_k\|^{2}=\|\bf w_k - \bf w\|^{2}=o(1)$. This completes the proof.
\end{proof}

\begin{proof}[Proof of Theorem \ref{thm:existence}] 
Since $\cN(B_1)\neq \emptyset$ by Lemma \ref{lem:nehari}, there is a sequence $(\bf u_k)$ in $\cN(B_1)$ such that $\cJ(\bf u_k)\to c(B_1)$. It follows from Theorem \ref{thm:minimizing} that there exists a least energy pinwheel solution $\bf w=(w_1,\ldots,w_\ell)$ to the system \eqref{eq:system_rn}. Then, $\bf u=(|w_1|,\ldots,|w_\ell|)\in \cN(\mathbb{R}^{N})$ and $\cJ(\bf u)=c(\rn)$. So $\bf u$ is a least energy pinwheel solution to \eqref{eq:system_rn} with non-negative components. 
\end{proof}

\section{Optimal pinwheel partitions for the Yamabe equation}
\label{sec:optimal_partition}

Let $G$ be the group defined in \eqref{eq:G}, and $\vr_\ell\in O(N)$ be as in \eqref{eq:rho}, i.e.,
$$\vr_\ell(z,y):=\Big(\Big(\cos\frac{\pi}{\ell}\Big)z+\Big(\sin\frac{\pi}{\ell}\Big)\tau z,\,y\Big),\qquad\text{if \ }(z,y)\in\cc^2\times\r^{N-4}.$$
Recall that an optimal $(G,\vr_\ell,\ell)$-partition for these particular choices is called an optimal $\ell$-pinwheel partition. To make this notion more precise, for any $G$-invariant open subset $\Omega$ of $\rn$ we set $D^{1,2}_0(\o)^G:=\{u\in D^{1,2}_0(\o):u\text{ is }G\text{-invariant}\}$ and consider the critical problem
\begin{equation} \label{probcrit}
	\begin{cases}
		-\Delta u = |u|^{2p-2}u, \\
		u\in D_0^{1,2}(\o)^{G},
	\end{cases}
\end{equation}
and its Nehari manifold $\mathcal{M}_{\Omega}:=\{ u\in D_0^{1,2}(\o)^{G}: u\neq 0, \ \|u\|^{2}=|u|_{2p}^{2p} \}$. Define
\begin{equation*}
	 \mathfrak{c}_{\o}=: \inf_{ u\in\mathcal{M}_{\Omega} } \frac{1}{N}\|u\|^{2}.
\end{equation*}

\begin{definition}
Let $G$ and $\vr_\ell$ be as above. An $\ell$-tuple $(\Omega_{1}, \dots, \Omega_{\ell})$ of subsets of $\rn$ is an \textbf{optimal $\ell$-pinwheel partition} for the Yamabe problem \eqref{eq:yamabe0} if
\begin{itemize}
\item[$(P_1)$] $\o_i\neq\emptyset$, \ $\o_i\cap\o_j=\emptyset$ if $i\neq j$, \ $\rn=\bigcup_{i=1}^\ell\overline{\o}_i$ \ and \ $\o_i$ is open and $G$-invariant,
\item[$(P_2)$] for each  $i=1,\ldots,\ell$, the value $c_{\o_i}$ is attained by some function $u_i\in\cM_{\o_i}$,
\item[$(P_3)$] $\vr_\ell(\o_2)=\o_1, \ \ldots \ , \ \vr_\ell(\o_\ell)=\o_{\ell-1}, \ \ \vr_\ell(\o_1)=\o_\ell$,
\item[$(P_4)$] and
\begin{equation*}
\sum_{j=1}^{\ell}\mathfrak{c}_{\o_{j}}=\inf_{(\Theta_{1},\dots \Theta_{\ell}) \in \mathcal{P}_{\ell}}	\sum_{j=1}^{\ell}\mathfrak{c}_{\Theta_{j}}.
\end{equation*}
where $\cP_\ell$ is the set of all $\ell$-tuples $(\Theta_{1},\dots \Theta_{\ell})$ of subsets of $\rn$ that satisfy $(P_1)$, $(P_2)$ and $(P_3)$.
\end{itemize}
\end{definition}

Let us define
\begin{equation*}
\mathcal{W}_\ell:=\{(u_{1},\dots,u_{\ell})\in \mathscr{D}(\rn):  u_{i}\neq 0, \ \|u_{i}\|^{2}=|u_{i}|_{2p}^{2p} \mbox{ for } i=1,\dots, \ell, \mbox{ and }u_{i}u_{j}=0  \mbox{ a.e. in } \mathbb{R}^{N} \mbox{ if } i\neq j  \}
\end{equation*}
and
\begin{equation*}
	\mathfrak{c}^{\ell}_{\infty}:=\inf_{(u_{1},\dots,u_{\ell})\in \mathcal{W}_{\ell} }\frac{1}{N}\sum_{i=1}^{\ell}\|u_{i}\|^{2}.
\end{equation*}
It is clear that $\mathcal{W}_{\ell}$ is not empty (see the proof of Lemma \eqref{lem:nehari}). Hence, the Sobolev inequality yields $\mathfrak{c}^{\ell}_{\infty}>0$. Note also that
\begin{equation}\label{desinf11}
	\mathfrak{c}_\infty^\ell\leq \inf_{(\Theta_{1},\ldots, \Theta_{\ell}) \in \mathcal{P}_{\ell}}	\sum_{j=1}^{\ell}\mathfrak{c}_{\Theta_{j}}.
\end{equation}

\begin{proof}[Proof of Theorem \ref{thm:main2}.]
Let $\bf u_k$ be as in the statement of Theorem \ref{thm:main2}. We write $\cJ_{k}:\cD(\rn)\to\r$ and $\cN_{k}(\rn)$ for the energy functional and the Nehari manifold associated to the system \eqref{eq:system} with $\beta=\beta_{k}$, as defined in \eqref{eq:energy} and \eqref{eq:nehari}, and we set
\begin{equation*}
c_{k}(\rn):=\inf_{\bf u\in\cN_{k}(\rn)}\cJ_{k}(\bf u)=\cJ_{k}(\bf u_{k}).
\end{equation*}
Since $\mathcal{W}_{\ell}\subset \cN_{k}(\rn)$, we have that
\begin{equation*}
\cJ_{k}(\bf u_{k})=\frac{1}{N}\sum_{i=1}^{\ell}\|u_{i,k}\|^{2}=c_{k}(\rn)\leq \mathfrak{c}^{\ell}_{\infty} \qquad\text{for every \ } k\in\mathbb{N}.
\end{equation*}
Fix $\delta\in\mathbb{R}$ such that $0<2\delta<S^{N/2}$. By Lemma \ref{lem:nehari} there exist $\eps_{k}\in(0,\infty)$ and $\zeta_k\in\rn$ such that
\begin{equation*}
\delta=\sup_{x\in\rn}\int_{B_{\eps_{k}}(x)}|u_{1,k}|^{2p}=\int_{B_{\eps_{k}}(\zeta_k)}|u_{1,k}|^{2p}.
\end{equation*}
Since $(u_{1,k})$ is bounded in $D^{1,2}(\rn)$, arguing as in the proof of Theorem \ref{thm:minimizing} we see that option $(ii)$ of Lemma \ref{lem:G} is impossible. Hence, there exist $C_0>0$ and $\xi_k\in\{0\}\times\r^{N-4}$ such that $\eps_k^{-1}|\zeta_k-\xi_k|<C_0$ and, as a consequence,
\begin{equation} \label{eq:xi2}
0<\delta=\sup_{x\in\rn}\int_{B_{\eps_{k}}(x)}|u_{1,k}|^{2p}\leq\int_{B_{(C_0+1)\eps_{k}}(\xi_k)}|u_{1,k}|^{2p}.
\end{equation}
Let $\bf w_{k}=(w_{1,k},\ldots,w_{\ell,k})$ be given by
\begin{equation*}
w_{i,k}(x):=\eps_{k}^{\frac{N-2}{2}}u_{i,k}(\eps_{k}x+\xi_k).
\end{equation*}
Since $\xi_k\in\{0\}\times\r^{N-4}$ we have that $\bf w_{k}\in \cN_{k}(\rn)$ and $\cJ_k(\bf w_{k})=\cJ_k(\bf u_{k})=c_k(\rn)$, so $\bf w_{k}$ is a least energy pinwheel solution to the system \eqref{eq:system_rn} with $\beta=\beta_k$.

$(i):$ \ After passing to a subsequence,
\begin{equation*}
\bf w_{k}\rightharpoonup \bf w_\infty \mbox{ weakly in } \mathscr{D}(\mathbb{R}^{N}),\qquad \bf w_{k}\rightarrow \bf w_\infty \mbox{ in }  (L^{2}_{\mathrm{loc}}(\mathbb{R}^{N}))^{\ell} \qquad \mbox{ and }\qquad \bf w_{k}\rightarrow \bf w_\infty  \mbox{ \ a.e. in } ( \mathbb{R}^{N}) ^{\ell}. 
\end{equation*}
Hence $\bf w_{\infty}=(w_{1,\infty},\dots, w_{\ell,\infty} )$ satisfies that $w_{i,\infty} \geq 0$ for every $i=1,\dots, \ell$. Equation \eqref{eq:xi2} yields
\begin{equation}\label{eq:xi3}
\delta=\sup_{x\in\mathbb{R}^{N}} \int_{B_{1}(x)}|w_{1,k}|^{2p}\leq 	\int_{B_{C_0+1}(0)}|w_{1,k}|^{2p}.
\end{equation}
So arguing as in the proof of Theorem \ref{thm:minimizing}, using \eqref{eq:xi3}, one shows that $\bf w_{\infty}=(w_{1,\infty},\dots, w_{\ell,\infty} )\neq\bf 0$. 

On the other hand, since $\partial_i\cJ_{k}(\bf w_{k})w_{i,k}=0$ for every $k\in\mathbb{N}$, we have
\begin{equation*}
0\leq -\beta_{k}\sum\limits_{\substack{j=1\\j\neq i}}^\ell\irn|w_{j,k}|^p|w_{i,k}|^{p}\leq \irn|w_{i,k}|^{2p}\leq C_{1}\quad\text{for all \ }k\in\n,
\end{equation*}
and using Fatou's Lemma we obtain 
\begin{equation*}
\irn|w_{j,\infty}|^p|w_{i,\infty}|^{p} \leq \liminf_{k\rightarrow \infty} \irn|w_{j,k}|^p|w_{i,k}|^{p}\leq \liminf_{k\rightarrow \infty} \frac{C_1}{-\beta_{k}}=0\qquad\text{for \ }i\neq j.
\end{equation*}
This implies that $w_{j,\infty}w_{i,\infty}=0$ a.e. in $\mathbb{R}^{N}$ whenever $i\neq j$. Furthermore, as
\begin{equation*}
	\begin{split}
		0&=\partial_i\cJ_{k}(\bf w_{k})w_{i,\infty}=\langle w_{i,k},w_{i,\infty}\rangle-\irn|w_{i,k}|^{2p-2}w_{i,k}w_{i,\infty}-\beta_{k}\sum\limits_{\substack{j=1\\j\neq i}}^\ell\irn|w_{j,k}|^p|w_{i,k}|^{p-2}w_{i,k}w_{i,\infty}\\
		&\geq \langle w_{i,k},w_{i,\infty}\rangle-\irn|w_{i,k}|^{2p-2}w_{i,k}w_{i,\infty},
	\end{split}
\end{equation*}
we have that
\begin{equation*}
 \langle w_{i,k},w_{i,\infty}\rangle\leq\irn|w_{i,k}|^{2p-2}w_{i,k}w_{i,\infty} \qquad\mbox{ for every \ } k\in\mathbb{N}
\end{equation*}
and, since $\bf w_{k}\rightharpoonup \bf w_{\infty} \mbox{ weakly in } \mathscr{D}(\mathbb{R}^{N})$, we obtain
\begin{equation*}
\|w_{i,\infty}\|^{2}\leq\irn|w_{i,\infty}|^{2p}.
\end{equation*}
Set
\begin{equation*}
	t:=\left( \frac{\|w_{1,\infty}\|^{2}}{\int_{\mathbb{R}^{N}}|w_{1,\infty}|^{2p} }\right)^{\frac{1}{2p-2}} 
\end{equation*} 
Note that $t\in \left( 0,1\right] $ and $t \bf w _{\infty} \in 	\mathcal{W}_{\ell}$. Therefore,
\begin{equation*}
 \mathfrak{c}^{\ell}_{\infty} \leq \frac{1}{N}\|t\bf w_{\infty}\|^{2}\leq \frac{1}{N}\|\bf w_{\infty}\|^{2}\leq \frac{1}{N}\liminf_{k\rightarrow \infty}\|\bf w_{k}\|^{2} \leq \mathfrak{c}^{\ell}_{\infty}.
\end{equation*}
Hence,  $t=1,$ \ $ \bf w _{\infty} \in 	\mathcal{W}_{\ell},$ \ $\bf w_{k}\rightarrow \bf w_{\infty}$ strongly in $\mathscr{D}(\mathbb{R}^{N})$ and
\begin{equation}\label{igu12}
	\frac{1}{N}\|\bf w_{\infty}\|^{2}=\mathfrak{c}^{\ell}_{\infty}.
\end{equation}
Therefore, $ w_{i,k}\rightarrow  w_{i,\infty}$ strongly in $L^{2p}(\mathbb{R}^{N})$ and
\begin{equation*}
\lim_{k\to\infty}\|w_{i,k} \|^{2}=\|w_{i,\infty}\|^{2}=|w_{i,\infty}|_{2p}^{2p}=\lim_{k\to\infty}|w_{i,k}|_{2p}^{2p} \ \ \mbox{ for every  } i=1,\ldots,\ell.
\end{equation*}
It follows that
\begin{equation*}
\lim_{k\rightarrow\infty}\beta_{k}\sum\limits_{\substack{j=1\\j\neq i}}^\ell\irn|w_{j,k}|^p|w_{i,k}|^{p}=\lim_{k\rightarrow\infty}(\|w_{i,k} \|^{2}-|w_{i,k} |_{2p}^{2p})=0 \quad\mbox{ whenever } i\neq j,
\end{equation*}
as claimed.
\medskip

$(ii):$ \ Lemma \ref{maincont} shows that $(w_{i,k})$ is uniformly bounded in $L^\infty(\rn)$. It follows from \cite[Theorem B.2]{cpt}  that  $(w_{i,k})$ is uniformly bounded in $\cC^{0,\alpha}(K)$ for any compact subset $K$ of $\mathbb{R}^{N}$ and $\alpha\in (0,1)$. So using the Arzelà-Ascoli theorem we conclude that $w_{i,\infty}\in \cC^{0}(\mathbb{R}^{N})$ for every $i=1,\dots,\ell$. As a consequence,
\begin{equation*}
\Omega_{i}:=\{x\in\mathbb{\mathbb{R}}^{N}: w_{\infty,i}(x)>0 \}
\end{equation*}  
is open and nonempty. Since $\bf w_{\infty}\in \mathscr{D}(\mathbb{R}^{N})$, each $w_{j,\infty}$ is $G$-invariant and $w_{j+1,\infty}=w_{j,\infty}\circ\vr_\ell$. These properties imply that each $\o_j$ is $G$-invariant and that $(\Omega_{1},\ldots, \Omega_{\ell})$ satisfies $(P_3)$. Furthermore, since $w_{j,\infty}w_{i,\infty}=0$, we have that $\Omega_{i}\cap \Omega_{j}=0$ whenever $i\neq j$.

As $ \bf w _{\infty} \in 	\mathcal{W}_{\ell}$, we have that $\|w_{i,\infty}\|^{2}=|w_{i,\infty}|_{2p}^{2p}$. Therefore $w_{i,\infty}\in \mathcal{M}_{\Omega_{i}}$ for every $i$, where $\mathcal{M}_{\Omega_{i}}$ is the  Nehari manifold for problem \eqref{probcrit} in $\o_i$. 
We claim that
\begin{equation*}
\frac{1}{N}\|w_{1,\infty}\|^2=\mathfrak{c}_{\o_{1}}:= \inf_{u\in\mathcal{M}_{\Omega_1}}\frac{1}{N}\|u\|^2.
\end{equation*}
Otherwise, there would exist $v_{1}\in \mathcal{M}_{\Omega_{1}}$ such that
\begin{equation*}
\frac{1}{N}\|w_{1,\infty}\|^{2}> \frac{1}{N}\|v_{1}\|^{2}\geq	\mathfrak{c}_{\o_{1}}.
\end{equation*}
Then, if we define $v_{j+1}:=v_j\circ\vr_\ell$, $j=1,\ldots,\ell-1$, we have that $\bf v=(v_{1},\dots,v_{j})\in \mathcal{W}_{\ell}$ and
\begin{equation*}
	\frac{1}{N}\sum_{i=1}^{\ell}\|w_{i,\infty}\|^{2}> \frac{1}{N}\sum_{i=1}^{\ell}\|v_{i}\|^{2}\geq	\mathfrak{c}^{\ell}_{\infty},
\end{equation*}
which contradicts equation \eqref{igu12}. This shows that $\frac{1}{N}\|w_{1,\infty}\|^{2}=	\mathfrak{c}_{\o_{1}}$ and, by symmetry, $\frac{1}{N}\|w_{i,\infty}\|^{2}=	\mathfrak{c}_{\o_i}$ for every $i$. Therefore $w_{i,\infty}$ is a least energy solution for problem \eqref{probcrit} in $\o_i$. 
\medskip

$(iii):$ \ As shown above, $(u_{i,k})$ is uniformly bounded in $\cC^{0,\alpha}(\Omega)$ for any bounded open subset $\Omega$ of $\mathbb{R}^{N}$ and $\alpha\in (0,1)$. Applying the Arzelà-Ascoli theorem, we deduce that $w_{i,k}\rightarrow w_{i,\infty}$ in $\cC^{0,\alpha}(\Omega)$ for all $\alpha\in (0,1)$. We now have all the hypotheses that are needed to derive the statements in $(iii)$ from \cite[Theorem C.1]{cpt}.
\medskip

$(iv):$ \ As a consequence of $(iii)$ we get that $\rn=\bigcup_{i=1}^\ell\overline{\o}_i$. This completes the proof of property $(P_1)$. Properties $(P_2)$ and $(P_3)$ were proved in $(ii)$. Finally, from equations \eqref{desinf11} and \eqref{igu12} we obtain
\begin{equation*}
\inf_{(\Theta_{1},\dots, \Theta_{\ell}) \in \mathcal{P}_{\ell}}	\sum_{j=1}^{\ell}\mathfrak{c}_{\Theta_{j}}\leq 	\sum_{i=1}^{\ell}\mathfrak{c}_{\o_{i}}=\sum_{i=1}^{\ell}\frac{1}{N}\|w_{i,\infty}\|^{2}=\mathfrak{c}^{\ell}_{\infty}\leq \inf_{(\Theta_{1},\dots, \Theta_{\ell}) \in \mathcal{P}_{\ell}}	\sum_{j=1}^{\ell}\mathfrak{c}_{\Theta_{j}},
\end{equation*}
which yields $(P_4)$. This shows that $ (\Omega_{1},\dots, \Omega_{\ell}) $ is an optimal $\ell$-pinwheel partition for the Yamabe problem \eqref{eq:yamabe0}. 
\medskip

$(v)$ Let $\ell=2$ and let $\Gamma_{2}$ be the subgroup of $O(N)$ generated by $G$ and $\vr_2$ and $\phi_2:\Gamma_{2}\rightarrow \mathbb{Z}_{2}$ be the homomorphism given by $\phi_{2}(g):=1$ for every $g\in G$ and $\phi_{2}(\vr_2):=-1$. The solutions to the Yamabe problem \eqref{eq:yamabe0} that satisfy 
\begin{equation}\label{lequiva}
	v(\gamma x)=\phi_{2}(\gamma)v(x) \qquad \text{for all \ } \gamma\in \Gamma_{2}, \ x\in\mathbb{R}^{N},
\end{equation}
are the critical points of the functional $J:D^{1,2}(\mathbb{R}^{N})^{\phi_{2}}\to\r$ given by 
$$J(v):= \frac{1}{2}\|v\|^2 - \frac{1}{2p}|v|_{2p}^{2p},$$
where $D^{1,2}(\mathbb{R}^{N})^{\phi_{2}}:=\{v\in D^{1,2}(\mathbb{R}^{N}): v\text{ satisfies }\eqref{lequiva} \}.$
The nontrivial critical points of $J$ belong to the Nehari manifold
\begin{equation*}
\mathcal{M}_{\infty}^{\phi_{2}}:=\{ v\in D^{1,2}(\mathbb{R}^{N})^{\phi_{2}}: v\neq 0, \ \|v\|^{2}=|v|_{2p}^{2p} \}
\end{equation*}
which is a natural constraint for $J$. There is a one-to-one correspondence $\mathcal{M}_{\infty}^{\phi_{2}}\rightarrow \mathcal{W}_{2}$ given by 
\begin{equation*}
v\mapsto (v^{+},-v^{-}),\qquad v^+:=\max\{v,0\}, \ v^-:=\min\{v,0\},
\end{equation*}
whose inverse is $(v_{1},v_{2})\mapsto v_{1}-v_{2}$, and one has that
$$J(v_1-v_2)=\frac{1}{N}(\|v_1\|^{2}+\|v_2\|^{2})\qquad\text{for every \ }(v_{1},v_{2})\in\cW_2.$$
Then, from \eqref{igu12} we get
\begin{equation*}
\begin{split}
	J(w_{1,\infty}-w_{2,\infty})&=\frac{1}{N}(\|w_{1,\infty}\|^{2}+\|w_{2,\infty}\|^{2})\\
	& =\mathfrak{c}_\infty^2:=\inf_{(v_{1},v_{2})\in \mathcal{W}_{2} }\frac{1}{N}(\|v_{1}\|^{2}+\|v_{2}\|^{2})	=\inf_{v\in \mathcal{M}_{\infty}^{\phi_{2}}} J(v).
\end{split}
\end{equation*}
This proves that $w_{1,\infty}-w_{2,\infty}$ is a solution to the Yamabe problem \eqref{eq:yamabe0} that possesses the minimum energy among all solutions that satisfy \eqref{lequiva}.
\end{proof}

The following result was used in the proof of Theorem \ref{thm:main2}. Its proof follows well known regularity arguments. We give it here for the sake of completeness.

\begin{lemma}\label{maincont}
Let $\beta_{k}<0$ and $(u_{1,k},\dots, u_{\ell,k} )$ be a solution to the system \eqref{eq:system} in $\Omega$ with $\beta=\beta_{k}$ such that $u_{i,k}\rightarrow u_{i,\infty}$ strongly in $ D_0^{1,2}(\Omega)$ for $i=1,\dots, \ell$. Then $(u_{i,k})$ is uniformly bounded in $L^{\infty}(\Omega)$.
\end{lemma}

\begin{proof}
	Let $s\geq 0$ and assume that $u_{i,k}\in L^{2(s+1)}(\Omega)$ for every $k\in\mathbb{N}$. Fix $M>0$ and define $\varphi_{i,k}:=u_{i,k}\min\{|u_{i,k}|^{2s},M^{2}\}$. Then,
	\begin{equation*}
		\nabla \varphi_{i,k} =\min \{|u_{i,k}|^{2s},M^{2} \}\nabla u_{i,k}+2s|u_{i,k}|^{2s}(\nabla u_{i,k})1_{A},
	\end{equation*}
	where $1_{A}$ is the characteristic function for the set $A:=\{x\in\Omega: |u_{i,k}(x)|^{s}< M \}$. Hence, $\varphi_{i,k}\in D_0^{1,2}(\Omega)$ and, since $(u_{1,k},\dots, u_{\ell,k} )$ solves \eqref{eq:system}, $\beta_k<0$ and $u_{i,k}\varphi_{i,k}\geq 0$, we get
	\begin{equation*}
\io\nabla u_{i,k}\cdot \nabla \varphi_{i,k} =  \io|u_{i,k}|^{2p-2}u_{i,k}\varphi_{i,k}+\sum\limits_{\substack{j=1\\j\neq i}}^\ell\beta_{k}\io|u_{j,k}|^p|u_{i,k}|^{p-2}u_{i,k}\varphi_{i,k}\leq \io|u_{i,k}|^{2p-2}u_{i,k}\varphi_{i,k}.
	\end{equation*}
	On the other hand, 
	\begin{equation*}
		\begin{split}
			\nabla u_{i,k}\cdot \nabla \varphi_{i,k}= \min \{|u_{i,k}|^{2s},M^{2} \}|\nabla u_{i,k} |^{2}+2s|u_{i,k}|^{2s}|\nabla u_{i,k}|^{2}1_{A}\geq  \min \{|u_{i,k}|^{2s},M^{2} \}|\nabla u_{i,k} |^{2}.
		\end{split}
	\end{equation*}	
Then, using the Sobolev inequality, we obtain
\begin{align*}
\Big(\int_{\Omega} \left| \min \{|u_{i,k}|^{s},M \} u_{i,k}\right|^{2p} \Big)^{\frac{2}{2p}}&\leq  C\int_{\Omega} \left| \nabla\left( \min \{|u_{i,k}|^{s},M \} u_{i,k}\right) \right|^{2}\\
&=C\int_{\Omega} \left| \left( \min \{|u_{i,k}|^{s},M \} \nabla u_{i,k} +s|u_{i,k}|^{s}(\nabla u_{i,k})1_{A} \right) \right|^{2}\\
&\leq  2C\left( \int_{\Omega}  \min \{|u_{i,k}|^{2s},M^{2} \} \left|\nabla u_{i,k}\right|^{2} + s^{2}\int_{\Omega} |u_{i,k}|^{2s}1_{A}\left| \nabla u_{i,k} \right|^{2}\right) \\
&\leq 2C(1+s^{2})\int_{\Omega}  \min \{|u_{i,k}|^{2s},M^{2} \} \left|\nabla u_{i,k}\right|^{2}\\
&\leq  2C(1+s^{2})\int_{\Omega}\nabla u_{i,k}\cdot \nabla\varphi_{i,k} \\
&\leq  2C(1+s^{2})\int_{\Omega}|u_{i,k}|^{2p-2}u_{i,k}\varphi_{i,k}.
\end{align*}
Let $K>0$ and set $\Omega_{i,K}=\{x\in \Omega: |u_{i,\infty}|^{2p-2}\geq K \}$ and $\Omega_{i,K}^{c}=\{x\in \Omega: |u_{i,\infty}|^{2p-2}< K \}$. Since $u_{i,k}\varphi_{i,k}\in L^p(\o)$ and $u_{i,k}\varphi_{i,k}\geq 0$, using the Hölder inequality we get
\begin{align*}
&\int_{\Omega}|u_{i,k}|^{2p-2}u_{i,k}\varphi_{i,k} \\
&\leq \int_{\Omega}\left(|u_{i,k}|^{2p-2}-|u_{i,\infty}|^{2p-2} \right) u_{i,k}\varphi _{i,k} +\int_{\Omega}|u_{i,\infty}|^{2p-2}  u_{i,k}\varphi _{i,k} \\
&\leq \int_{\Omega}\left(|u_{i,k}|^{2p-2}-|u_{i,\infty}|^{2p-2}\right) u_{i,k}\varphi _{i,k} +\int_{\Omega_{i,K}}|u_{i,\infty}|^{2p-2}  u_{i,k}\varphi _{i,k}+\int_{\Omega_{i,K}^{c}}|u_{i,\infty}|^{2p-2}  u_{i,k}\varphi _{i,k} \\
&\leq \Big( \int_{\Omega}\left| |u_{i,k}|^{2p-2}-|u_{i,\infty}|^{2p-2}\right|^{\frac{p}{p-1}}\Big)^{\frac{p-1}{p}}\left| u_{i,k}\varphi _{i,k}\right|_{p}+\Big(\int_{\Omega_{i,K}} |u_{i,\infty}|^{2p} \Big)^{\frac{p-1}{p}} \left| u_{i,k}\varphi _{i,k}\right|_{p} +K\int_{\Omega}  u_{i,k}\varphi _{i,k},
\end{align*}
where $| \ \ |_p$ denotes the norm in $L^p(\o)$. Since $u_{i,\infty}\in L^{2p}(\Omega)$, we may fix $K$ large enough so that	
\begin{equation*}
	\Big(\int_{\Omega_{i,K}} |u_{i,\infty}|^{2p} \Big)^{\frac{p-1}{p}} <\frac{1}{8C(1+s^{2})}.
\end{equation*}
Moreover, since  $u_{i,k}\rightarrow u_{i,\infty}$ strongly in $ L^{2p}(\Omega)$, there exists $k_{0}$ such that
\begin{equation*}
	\left( \int_{\Omega}\left| |u_{i,k}|^{2p-2}-|u_{i,\infty}|^{2p-2}\right|^{\frac{p}{p-1}}\right)^{\frac{p-1}{p}}\leq C_{1}	 \left( \int_{\Omega}\left| u_{i,k}-u_{i,\infty}\right|^{2p}\right)^{\frac{p-1}{p}} \leq \frac{1}{8C(1+s^{2})}.
\end{equation*}	
for any $k>k_{0}$. Thus,
\begin{align*}
\left| u_{i,k}\varphi _{i,k}\right|_{p} &=\Big(\int_{\Omega} \left| \min \{|u_{i,k}|^{s},M \} u_{i,k}\right|^{2p} \Big)^{\frac{1}{p}}\\
&\leq \frac{1}{4}\left| u_{i,k}\varphi _{i,k}\right|_{p}+ \frac{1}{4}\left| u_{i,k}\varphi _{i,k}\right|_{p} + 2C(1+s^{2})K\int_{\Omega}  u_{i,k}\varphi _{i,k}.
\end{align*}
Hence,
\begin{align*}
\frac{1}{2}\left(\int_{\Omega}|u_{i,k}\min\{|u_{i,k}|^{2s},M^{2}\}u_{i,k}|^{p}\right)^{\frac{1}{p}} &=\frac{1}{2}\left| u_{i,k}\varphi _{i,k}\right|_{p}\leq  2K C(1+s^{2})\int_{\Omega}u_{i,k}\varphi _{i,k}\\
&=  2K C(1+s^{2})\int_{\Omega}  \min\{|u_{i,k}|^{2s},M^{2}\}u_{i,k}^{2}.
\end{align*}
Letting $M\rightarrow \infty$ we derive that
\begin{equation}
\frac{1}{2}\Big( \int_{\Omega}  |u_{i,k}|^{2p(s+1)}\Big)^{\frac{1}{p}}
\leq 2K C(1+s^{2})\int_{\Omega} |u_{i,k}|^{2(s+1)}
\leq  2K C(1+s)^{2}\int_{\Omega} |u_{i,k}|^{2(s+1)}.\label{s1}
\end{equation}

Now, to obtain the uniform bound in $L^\infty(\rn)$, we argue as in \cite[Lemma 3.2]{tt} (see also \cite[Lemma A.1]{css2}). First, let $C_1$ be such that $\int_{\rn} |u_{i,k}|^{2p} \leq C_1$ and let $C_0:=4 K C$. Then, using $s+1=p$, estimate \eqref{s1} yields that
\begin{equation*}
\int_{\Omega}  |u_{i,k}|^{2p^2}
\leq C_0^p p^{2p}C_1^p=:C_2.
\end{equation*}
Next, using $s+1=p^2$, estimate \eqref{s1} yields that
\begin{equation*}
\int_{\Omega}  |u_{i,k}|^{2p^3}
\leq C_0^p p^{4p} C_2^p
=C_0^{p+p^2} p^{2(p^2+2p)} C_1^{p^2}
=:C_3.
\end{equation*}
Now, iterating this procedure with $s_n+1=p^n$ for $n=1,\ldots,m$ for $m\in \mathbb N$, we obtain that
\begin{equation*}
\int_{\Omega}  |u_{i,k}|^{2p^m}
\leq C_0^{\sum_{n=1}^m p^n} p^{2\sum_{n=1}^m (m-n+1) p^{n}} C_1^{p^m}=:C_m.
\end{equation*}
Note that
\begin{align*}
 \frac{1}{p^m}\sum_{n=1}^m p^n
 =\frac{1}{p^m}\left(\frac{p \left(p^m-1\right)}{p-1}\right)
 =\frac{p \left(1-p^{-m}\right)}{p-1}\to \frac{p}{p-1}
 \qquad \text{ as }m\to \infty.
\end{align*}
And, using that $\sum_{n=1}^m np^{n-1}=\partial_p \sum_{n=1}^m p^n$,
\begin{align*}
 \frac{1}{p^m}\sum_{n=1}^m (m-n+1) p^{n}
 =\frac{1}{p^m}\left( \frac{p \left(p^{m+1}-m p+m-p\right)}{(p-1)^2}\right)
 =\frac{(m-(m+1) p) p^{1-m}+p^2}{(p-1)^2}\to \frac{p^2}{(p-1)^2}
\end{align*}
as $m\to \infty$. Then,
\begin{align*}
\sup_\Omega|u_{i,k}|
=\lim_{m\to \infty}\left(
\int_\Omega|u_{i,k}|^{2p^m}\right)^\frac{1}{2p^m} \leq C_m^{\frac{1}{2p_m}}
\to
C_0^{\frac{p}{2(p-1)}} p^{\frac{p^2}{(p-1)^2}} C_1^{\frac{1}{2}}<\infty,
\end{align*}
as claimed.
\end{proof}

\section{Nodal solutions for the Yamabe problem}
\label{sec:nodal}

In this section we prove our last main result.

\begin{proof}[Proof of Theorem \ref{thm:nodal}]
$(i):$ \ Let $G$ act on $\rn$ as in \eqref{eq:G} and $\vr_\ell$ be as defined in \eqref{eq:rho}. Let $\Gamma_\ell$ be the subgroup of $O(N)$ generated by $G\cup\{\vr_\ell\}$ and $\phi_\ell:\Gamma_\ell\to\z_2:=\{1,-1\}$ be the homomorphism of groups given by $\phi_\ell(g):=1$ if $g\in G$ and $\phi_\ell(\vr_\ell):=-1$. Since $\vr_\ell$ has order $2\ell$, this homomorphism is well defined. The $\Gamma_\ell$-orbit of a point $x\in\rn$, defined as
$$\Gamma_\ell x:=\{\gamma x:\gamma\in\Gamma_\ell\},$$
is infinite if $x=(z,y)\in\cc^2\times\r^{N-4}$ with $z\neq 0$, while $(0,y)$ is a fixed point of $\Gamma_\ell$ for every $y\in\r^{N-4}$. Then, by \cite[Corollary 3.4]{c}, there exists a sign-changing solution $\widehat{w}_\ell$ of the Yamabe problem \eqref{eq:yamabe0} such that
\begin{equation*}
\widehat{w}_\ell(gx)=\widehat{w}_\ell(x) \quad\text{and}\quad \widehat{w}_\ell(\vr_{\ell} x)= -\widehat{w}_\ell(x) \qquad\text{for all \ }g\in G, \ x\in\rn,
\end{equation*}
with $G$ as in \eqref{eq:G}, and $\widehat{w}_\ell$ has least energy among all nontrivial solutions to \eqref{eq:yamabe0} with these properties.

$(ii):$ \ To prove this statement if suffices to show that, if $\ell=nm$ with $n$ even and $u$ and $v$ are nontrivial functions that satisfy
$$u(\vr_\ell x)=-u(x)\quad\text{and}\quad v(\vr_m x)=-v(x)\quad\text{for every \ }x\in\rn,$$
then $u\neq v$. Indeed, arguing by contradiction, assume that $u=v$ and choose $x_0\in\rn$ such that $u(x_0)=v(x_0)\neq 0$. Then, since $\vr_\ell^n=\vr_m$ and $n$ is even, we have that 
$$u(\vr_mx_0)=u(\vr_\ell^nx_0)=u(x_0)=v(x_0)=-v(\vr_mx_0)=-u(\vr_mx_0).$$
This is a contradiction.
\end{proof}

\begin{remark}
\emph{Let $\o_1:=\{x\in\rn:\widehat{w}_\ell(x)>0\}$ and $\o_2:=\{x\in\rn:\widehat{w}_\ell(x)<0\}$. It is easy to verify that $(\o_1,\o_2)$ is an optimal $(K_\ell,\vr_\ell,2)$-partition for the Yamabe equation \eqref{eq:yamabe0}, where $K_\ell:=\ker\phi_\ell$ is the subgroup generated by $G\cup\{\vr_\ell^2\}$.}
\end{remark}

\subsection*{Acknowledgments}

J. Faya and A. Saldaña thank the Instituto de Matemáticas - Campus Juriquilla for the kind hospitality.

\bigskip

\begin{flushleft}
\textbf{Mónica Clapp}\\
Instituto de Matemáticas\\
Universidad Nacional Autónoma de México \\
Campus Juriquilla\\
76230 Querétaro, Qro., Mexico\\
\texttt{monica.clapp@im.unam.mx} 
\medskip

\textbf{Jorge Faya}\\
Instituto de Ciencias Fisicas y Matem\'aticas\\
Universidad Austral de Chile\\
Facultad de Ciencias\\
Av. Rector Eduardo Morales Miranda 23\\
Valdivia, Chile\\
\texttt{jorge.faya@uach.cl}
\medskip

\textbf{Alberto Saldaña}\\
Instituto de Matemáticas\\
Universidad Nacional Autónoma de México \\
Circuito Exterior, Ciudad Universitaria\\
04510 Coyoacán, Ciudad de México, Mexico\\
\texttt{alberto.saldana@im.unam.mx}
\end{flushleft}
	
%\includepdf[pages=-]{1.pdf}
\end{document}